\documentclass[11pt]{amsart}

\usepackage[T1]{fontenc}
\usepackage{lmodern}
\usepackage[protrusion=true,expansion=false]{microtype}
\usepackage{amsmath,amssymb,amsthm,mathtools}
\usepackage[margin=2.5cm]{geometry}
\usepackage[colorlinks=true,linkcolor=blue,citecolor=blue,urlcolor=blue,
            hypertexnames=false]{hyperref}
\usepackage{setspace}

\newtheorem{mainresult}{Theorem}

\newtheorem{maincorollary}[mainresult]{Corollary}
\newtheorem{theorem}{Theorem}[section]
\newtheorem{proposition}[theorem]{Proposition}
\newtheorem{lemma}[theorem]{Lemma}

\theoremstyle{remark}

\allowdisplaybreaks

\newcommand{\R}{\mathbb R}
\newcommand{\C}{\mathbb C}
\newcommand{\T}{\mathbb T}
\newcommand{\HH}{\mathcal H}
\newcommand{\Sc}{\mathcal S}
\newcommand{\supp}{\operatorname{supp}}

\newcommand{\D}{\mathcal D}
\newcommand{\1}{\mathbf 1}
\newcommand{\loc}{\mathrm{loc}}
\title{Phase retrieval for Schr\"odinger evolutions}

\author{Ben Pineau}
\address{Courant Institute School of Mathematics, Computing, and Data Science, New York University,
New York, NY, 10012}
\email{brp305@nyu.edu}

\author{João P. G. Ramos}
\address{Instituto de Matemática Pura e Aplicada (IMPA) - Estrada Dona Castorina 110, 22460-320, Rio de Janeiro - RJ, Brazil.}
\email{joao.ramos@impa.br}

\author[Mitchell A. Taylor]{Mitchell A. Taylor}
\address{Mathematical Institute, University of Oxford, Andrew Wiles Building, Radcliffe Observatory Quarter, Woodstock Road, Oxford, OX2 6GG, United Kingdom}
\email{mitchtaylor@shaw.ca}

\begin{document}

\begin{abstract}
Consider the one-dimensional Schr\"odinger evolution
\begin{equation}\label{SchV}
i\partial_tw+\partial_x^2w=Vw,\hspace{5mm}(t,x)\in [0,T]\times\mathbb{R}   , 
\end{equation}
where $V$ is a real-valued time-dependent potential belonging to $L^1_tL^\infty_x+L^2_{t,x}$. We prove that an $L^2$ solution to the above equation is determined by its mass density $|w|^2$ modulo a global unimodular phase factor. This answers a question recently posed by Jaming \cite{Jaming2025} for a large class of  potentials, and recovers his earlier result for the free Schr\"odinger evolution. Our proof is simple and robust. It combines the basic local smoothing properties of the above Schr\"odinger evolutions, Strichartz estimates, and a stripwise version of a unique continuation theorem of Ionescu and Kenig to show that the interaction function
\begin{equation*}
F(t,x,y)=u(t,x)v(t,y)-u(t,y)v(t,x)    
\end{equation*}
vanishes identically on $[0,T]\times\mathbb{R}^2$ whenever $u$ and $v$ are $L^2$ solutions to \eqref{SchV} with the same mass density. The method immediately extends to several important nonlinear models such as the cubic NLS equation, which is globally well-posed in $L^2(\mathbb{R})$ with $V=|u|^2$. Interestingly, we also show that the above phase recovery property fails in dimensions two and higher, even for the free equation  with Schwartz initial data.
\end{abstract}

\maketitle

\section{Introduction}

Let $\Omega\subset\R^n$ be a domain and let $T>0$. We say that a function
$w:[0,T]\times\Omega\to\C$ solves a Schr\"odinger equation with real-valued potential
$V:[0,T]\times\Omega\to\R$ if it satisfies
\begin{equation}\label{eq:schr-pot}
i\partial_t w+\Delta w=Vw,
\end{equation}
in a suitable sense, depending on the regularity of the solution and the potential. Equation \eqref{eq:schr-pot} is of pivotal importance in quantum mechanics, where by adjusting the potential $V$ we may use it to describe the time evolution of the spatial distribution of particles. We refer the reader to Teschl~\cite{TeschlBook} for the theoretical background on the physics of this equation.

The main question we are concerned with in this manuscript is heavily motivated by quantum mechanics, and dates back to the work of Pauli~\cite{Pauli}. Indeed, let us consider first the problem of determining when it is possible to recover a (nice enough) function $f$ up to a unimodular constant from measurements of the form $|f|$, $|\widehat f|$. This is the so-called \emph{Pauli problem}~\cite{CorbettHurst,RamosSousa,Vogt}. From a quantum mechanical point of view, this can be thought of as recovering a wave function from its probability densities of position and momentum at a single time. By now, it is well-known that this is not possible. 

On the other hand, Pauli's problem is intimately connected to the problem of phase retrieval for Schr\"odinger equations. Indeed, if $u\in C(\mathbb{R};L^2(\mathbb{R}^n))$ is a solution to the free Schr\"odinger evolution,
\begin{equation}\label{eq:free-schr}
\begin{cases}
i\partial_t u(t,x)+\Delta u(t,x)=0 & \text{in }\R\times\R^n,\\
u(0,x)=f(x) & \text{on }\R^n,
\end{cases}
\end{equation}
then, for $t\ne0$, the solution can be represented as
\[
u(t,x)=\frac{1}{(4\pi i t)^{n/2}}e^{i|x|^2/(4t)}
\widehat{\left(e^{i|\cdot|^2/(4t)}f\right)}
\left(\frac{x}{4\pi t}\right),
\]
where $\widehat h(\xi)=\int_{\R^n}e^{-2\pi i x\cdot\xi}h(x)\,dx$ is the (suitably normalized) Fourier transform on $\mathbb{R}^n$. In other words, if we are able to solve Pauli's problem for the corresponding chirped datum, we would be able to recover the solution $u$ to the free Schr\"odinger equation \eqref{eq:free-schr} above from its values $|u(0,\cdot)|$, $|u(t_0,\cdot)|$ for some $t_0\neq 0$, and vice-versa. Regarding the problem in this particular setting, the counterexamples to Pauli's problem do show that recovery of the solution from two times is not possible in general. This has led the community to several follow-up questions, such as Wright's conjecture~\cite{JamingRathmair}, which posits that there exists a third unitary operator $T$ defined on $L^2$ such that any $f \in L^2(\R)$ may be recovered from $|f|,|Tf|, |\widehat f|$. Another closely related question~\cite{CarmeliHeinosaariSchultzToigo} is that of determining the minimal $N$ such that there are $N$ times $\{t_i\}_{i=1}^N$ for which any solution to \eqref{eq:free-schr} is determined up to multiplication by scalars from $\{|u(t_i,x)|:x\in\R,\ i=1,\dots,N\}$. To the best of our knowledge, both problems remain open.

The Pauli problem can also be thought of as one instance of the broader phase retrieval problem: recovering a function or vector, up to a constant phase, from magnitudes of prescribed measurements. This problem also arises in optics, crystallography and signal processing; we refer to~\cite{GrohsKoppensteinerRathmair,KlibanovSacksTikhonravov} for surveys. Finite-dimensional recovery from frame coefficients in deterministic and probablisitc settings are studied in~\cite{BalanCasazzaEdidin,CandesStrohmerVoroninski}, while infinite-dimensional Hilbert- and Banach-space formulations are treated in~\cite{AlaifariGrohs,CahillCasazzaDaubechies}. Stability on subspaces of function spaces is investigated in~\cite{AbdallaDeDiosRamosTaylor,ChristPineauTaylor,FreemanOikhbergPineauTaylor}, and H\"older stability on nonlinear subsets is studied in~\cite{FreemanTaylor}. For short-time Fourier measurements, local stability and its obstructions are considered in~\cite{AlaifariDaubechiesGrohsYin,AlaifariPineauTaylorWellershoff,BertoliniDeDiosPineauRamosTaylor,GrohsRathmair};  stability for squares in the Fock space is treated in~\cite{BortolottoRamos}.

Back to the particular realm of Schr\"odinger equations, one is led to ask more generally if one may recover the solution $w$ to \eqref{eq:schr-pot} from the data $|w(t,\cdot)|$, where $t$ runs through a set of times. In this direction, Jaming~\cite{Jaming2014} first showed that if one considers solutions to \eqref{eq:free-schr}, then phase retrieval is possible from $\{|u(t,x)|\}_{(t,x)\in\R\times\R}$ by connecting the problem to the fractional Fourier transform. Moreover, he later conjectured~\cite{Jaming2025} that the same should hold for \eqref{eq:schr-pot} under certain conditions on the potential $V$, and suggested that the same result should hold in higher dimensions.

The main purpose of this manuscript is to answer some of these latter questions on phase retrieval for solutions to \eqref{eq:schr-pot}. Our main result is the following generalization of Jaming's result which holds for a large class of real-valued potentials:

\begin{mainresult}
\label{thm:time-potential}
Let $I_0\subset\R$ be a nonempty open interval. Suppose that the measurable
function $V:I_0\times\R\to\R$ admits, on every compact subinterval
$J\subset I_0$, a decomposition of the form
\[
 V=V_\infty+V_2,
 \qquad
 V_\infty\in L^1(J;L^\infty(\R)),
 \qquad
 V_2\in L^2(J\times\R).
\]
Let $u,v\in C(I_0;L^2(\R))$ be mild solutions of \eqref{eq:schr-pot}. If
\[
 |u(t,x)|=|v(t,x)|
 \quad\text{for almost every }(t,x)\in I_0\times\R,
\]
then $v=\zeta u$ on $I_0\times\R$ for some
$\zeta\in\T$.
\end{mainresult}
Here, a solution $w$ to \eqref{eq:schr-pot} is said to be ``mild'' if it satisfies the Duhamel formula. That is, for $t_0,t\in I_0$, we have
\begin{equation}\label{milddef}
 w(t)=e^{i(t-t_0)\partial_x^2}w(t_0)
 -i\int_{t_0}^t e^{i(t-s)\partial_x^2}(V(s)w(s))\,ds
\end{equation}
in the sense of distributions. On every compact interval $J\subset I_0$, the products of $w$ with each part of the potential satisfy
\[
 V_\infty w\in L^1(J;L^2),
 \qquad
 V_2w\in L^2(J;L^1)\subset L^{4/3}(J;L^1),
\]
by H\"older's inequality and the fact that $w\in C(J;L^2)$. The Duhamel integral above is then interpreted using this decomposition. The motivation for this particular potential decomposition comes from the fact that a $C(J;L^2(\mathbb{R}))$ solution to \eqref{milddef} satisfies the natural Strichartz estimate (see Lemma \ref{lem:endpoint-wronskian} below) $w\in L_t^{\infty}L_x^2\cap L_t^4L_x^{\infty}$; the assumption on the potential then ensures that we can measure the source term $Vw$ in the natural dual norm $L_t^1L_x^2+L_t^{\frac{4}{3}}L_x^1$.

We also remark that, in particular, the potentials in~\cite{Jaming2025} are all assumed to be time-\emph{independent}, and hence Theorem \ref{thm:time-potential} provides a positive answer to Question I.2 in~\cite{Jaming2025} for a wide class of stationary potentials; namely it applies to any potential in $L^2(\R)+L^\infty(\R)$, which includes all $L^p(\R)$,
$2\le p\le\infty$.

In contrast to ~\cite{Jaming2025}, our proof relies essentially only on physical space tools and basic estimates from the theory of dispersive PDE. To briefly sketch the proof of Theorem \ref{thm:time-potential}, we start with the standard endpoint Strichartz and local smoothing estimates~\cite{KeelTao,VegaVisciglia}. On each compact time interval,  we have
\[
 u,v\in L^4_tL^\infty_x,
 \qquad
 \chi u,\chi v\in L^2_tH^{1/2}_x
 \quad(\chi\in C_c^\infty(\R)).
\]
The latter property provides just barely enough regularity to justify the continuity equation
\begin{equation}\label{localmass}
\partial_t|w|^2+2\partial_x\operatorname{Im}(\overline w\,\partial_xw)=0
\end{equation}
in the sense of distributions. For Schwartz solutions $u$ and $v$ with $|u|=|v|$ this will show that the currents (or momentum densities) agree in the sense of distributions, 
\begin{equation*}
\operatorname{Im}(\overline{u}\partial_xu)=\operatorname{Im}(\overline{v}\partial_xv)    
\end{equation*}
and thus, since $2\operatorname{Re}(\overline{u}\partial_xu)=\partial_x|u|^2$, we would have
\begin{equation*}
\overline{u}\partial_xu-\overline{v}\partial_xv=0.    
\end{equation*}
Using $|u|=|v|$, this formally gives
\begin{equation*}
v\,\partial_xu-u\,\partial_xv=0
 \quad\text{in }\D'(\R_x)    
\end{equation*}
for almost every time. One of the subtle difficulties in our argument will be in showing that this can be made sense of for solutions with merely $L^2$ initial data. Once this is achieved, the next step is to introduce the exterior product, or ``interaction function"
\[
 F(t,x,y)=u(t,x)v(t,y)-v(t,x)u(t,y).
\]
Our goal is to show that this vanishes identically, which would prove, in particular, that the two solutions $u$ and $v$ are linearly dependent. In order to do so, we make the observation that $F$ also solves a two-dimensional Schr\"odinger equation with potential $V(t,x)+V(t,y)$. In the coordinates $r=(x+y)/\sqrt2$ and $s=(x-y)/\sqrt2$, the function $F$ is odd in $s$. Its Dirichlet trace (interpreted in a certain weak sense) at $s=0$ vanishes by oddness, and its normal trace vanishes by the Wronskian identity above. These vanishing Dirichlet and normal traces will enable us to show that the zero extension of $F$ from $s>0$ solves the same two-dimensional Schr\"odinger equation without a singular boundary source term.

Finally, a stripwise version of the Ionescu--Kenig Carleman estimate~\cite{IonescuKenig} (which applies to our potentials) propagates this zero half-space through the whole $(r,s)$-plane to show that $F$ vanishes identically. Thus, $u$ and $v$ are linearly dependent at one time, and equality of their moduli makes the proportionality constant unimodular, with uniqueness propagating the same phase throughout $I_0$.

If the potential and both solutions extend globally, uniqueness propagates the conclusion to every time.  The theorem also applies to nonlinear equations whenever equality of the moduli makes their nonlinear potentials
identical. This gives immediately the following consequence for the cubic nonlinear Schr\"odinger equation (or cubic NLS for short).

\begin{maincorollary}
\label{cor:cubic-NLS}
Let $\sigma\in\R$, let $I_0\subset\R$ be a nonempty open interval, and let
$u,v\in C(\R;L^2(\R))$ be global mild solutions of
\[
i\partial_t w+\partial_x^2w=\sigma|w|^2w.
\]
If $|u|=|v|$ almost everywhere on $I_0\times\R$, then
$v=\zeta u$ for every time and some $\zeta\in\T$.
\end{maincorollary}
Here, cubic NLS mild solutions are understood in the standard class
$C(\R;L^2(\R))\cap L^8_{\loc}(\R;L^4(\R))$.

As a final result, we prove that the prediction made by Jaming in~\cite{Jaming2025} regarding the extendability of the results for \eqref{eq:free-schr} to higher dimensions is \emph{false}.
\begin{mainresult}
\label{thm:higher-dimensional-failure}
For every $d\ge2$ there are nonproportional functions
$f,g\in\Sc(\R^d)$ such that
\[
 |e^{it\Delta}f(x)|=|e^{it\Delta}g(x)|
 \quad\text{for every }(t,x)\in\R\times\R^d.
\]
\end{mainresult}

The counterexample is based on the Gaussian $\phi(x)=e^{-x^2/2}$ and its negative derivative $\psi(x)=xe^{-x^2/2}$. In dimension two, set
\[
 f=\phi\otimes\psi+i\psi\otimes\phi,
 \qquad
 g=\phi\otimes\psi-i\psi\otimes\phi.
\]
These are nonproportional Schwartz functions. If $A_t=e^{it\partial_x^2}\phi$, then $e^{it\partial_x^2}\psi=xA_t/(1+2it)$, and the tensor-product factorization of the free propagator gives
\begin{align*}
 e^{it\Delta}f(x_1,x_2)
 &=\frac{A_t(x_1)A_t(x_2)}{1+2it}(x_2+ix_1),\\
 e^{it\Delta}g(x_1,x_2)
 &=\frac{A_t(x_1)A_t(x_2)}{1+2it}(x_2-ix_1).
\end{align*}
The above two expressions have the same modulus for every point and every time. To extend the result to $d>2$, simply tensor both functions with a Gaussian in the remaining variables.
\\

It is instructive to interpret this counterexample in relation to the general continuity equation
\begin{equation*}
\partial_t|w|^2+2\nabla\cdot \operatorname{Im}(\overline{w}\nabla w)=0  
\end{equation*}
which is the higher dimensional analogue of \eqref{localmass}. In one dimension, one sees formally that the evolution of the mass density $|w|^2$ uniquely determines the momentum density (for regular enough sufficiently decaying data), but it only determines the momentum density modulo divergence free vector fields in higher dimensions. In the two dimensional example above, the initial momentum densities are
\[
p_f(x)=-p_g(x)=e^{-|x|^2}(x_2,-x_1),
\]
which are purely rotational and divergence-free. Thus, the two data carry
opposite circulations, while their evolutions have identical mass densities.
\\

This article is organized as follows. In Section 2, we prove Theorem \ref{thm:time-potential}. In Section 3, we deduce Corollary \ref{cor:cubic-NLS} as a consequence of the preceding result. Section 4 is then dedicated to the counterexample construction of Theorem \ref{thm:higher-dimensional-failure}. Finally, we gather, in Section 5 below, some other results obtained in parallel to the ones here. These are results which are mostly superseded/covered by the ones in the present paper, but since their techniques (we believe) are interesting in their own right, we decided to keep them in a separate repository for independent consultation by the interested reader.

\subsection*{LLM Usage.} Large language models have been of great importance in the production of this manuscript. Below is a brief narration of how they contributed to the current work. 

We started from the first idea of considering the continuity equation $ \partial_t|w|^2+2\partial_x\operatorname{Im}(\overline w\,\partial_xw)=0$, together with integration by parts, in order to show that, for two sufficiently regular solutions $u,v$ with $|u|=|v|$, a smooth relative phase in the representation $v(t,x)=e^{i\phi(t,x)}u(t,x)$ must be locally constant in space on the nonvanishing set of $u$. The equations then imply that the relative phase is constant on each connected component of this nonvanishing set. Our initial argument then ran through classical works (see e.g. \cite{Masuda}) to obtain uniqueness of the phase overall.

This argument, however, only works in high enough regularity. In order to remove such regularity, we needed to get rid of the assumptions on the phase. Entirely human ideas allowed us to work with data in $L^2$ for a restricted class of potentials, and data in $H^{\epsilon}$ for a large class of potentials, but not both simultaneously.  At that point, we started exploring with the aid of GPT 5.4/Claude Opus 4.6. It gave us several different directions, many of which are described in the end of this manuscript, such as ones based on scattering theory. 

By exploiting and exhausting several paths towards the general $L^2$ case with GPT 5.4 and 5.5, as well as Claude Opus 4.7 and 4.8, we had the idea of working at the level of currents. By doing so, we were naturally led to the idea of lifting our equation one dimension higher. The details on how to finish from there, such as applying specific unique continuation lemmas, were then provided by interacting with several different LLMs. Finally, GPT 5.5 in Codex was also used in order to obtain a Lean verification of Corollary \ref{cor:cubic-NLS}.

Most of the other attempts and results through different methods have been written down and are kept \href{https://github.com/joaopgramos95/Schrodinger-PR/tree/main}{here}, together with the aforementioned Lean verification. Apart from that note, the results and proofs here have all been written and verified by humans. We take full responsibility for the contents of both manuscripts. 

\section*{Acknowledgements}

The authors are thankful to Philippe Jaming, Jaume de Dios and Lukas Liehr for several enlightening conversations about the topic during the elaboration of the manuscript. 

J. P. G. R.  was supported by the FCT through project SHADE
(2023.17881.ICDT), DOI \allowbreak \texttt{10.54499/}
\texttt{2023.17881.ICDT}. He was also supported  by FAPERJ through  JCNE grant
no.~SEI-260003/ \allowbreak 020475/2025, and by Instituto Serrapilheira through
grant Serra-R-2510-59765.

\section{Proof of Theorem \ref{thm:time-potential}}

We start with the following local form of the half-space unique continuation argument of
Ionescu--Kenig~\cite{IonescuKenig}. Throughout the proof of the following proposition, we borrow for ease of comparison, the notation used in this paper.

\begin{proposition}
\label{prop:strip-propagation}
Let $J$ be a compact interval of positive length, $z=(z_1,z_2)\in\R^2$.
Put $\HH=i\partial_t+\Delta_z$ and set
\begin{align*}
 X&=L^1(J;L^2(\R^2))+L^{4/3}(J\times\R^2),\\
 X'&=L^\infty(J;L^2(\R^2))\cap L^4(J\times\R^2),\\
 Y&=L^1(J;L^\infty(\R^2))+L^2(J\times\R^2).
\end{align*}
Then there is an absolute $\varepsilon_*>0$ with the following
property.  Suppose
$Z\in C(J;L^2)\cap X'$ is such that
\[
i\partial_tZ+\Delta_zZ=WZ\in X,
\]
and $Z=0$ on $J\times\{z_1>b\}$.  Here, $W$ is a potential such that, for some $h>0,$
\[
 \sup_{a\le b-h}
 \|\1_{\{a<z_1<a+h\}}W\|_Y\le\varepsilon_*.
 \tag{1.3}\label{eq:strip-smallness}
\]
Under these hypotheses, we have that $Z=0$ on $J\times\R^2$.
\end{proposition}

\begin{proof}
We will prove that the zero set can be enlarged by one strip, and then
repeat that argument.  The analytic input is the Carleman inequality in
\cite[Lemma~9.1]{IonescuKenig}.  We state the form we use and verify its
hypotheses below.

\medskip
\noindent\emph{Step 1: the multiplication estimate.}
We begin by recording a simple product estimate for the space $X$. We first recall that the  norm on $Y$ is
\[
 \|A\|_Y=\inf_{A=A_\infty+A_2}
 \bigl(\|A_\infty\|_{L^1_tL^\infty_z}
       +\|A_2\|_{L^2_{t,z}}\bigr).
\]
The norm on $X$ is defined in the same way using its two summands, while
\[
 \|B\|_{X'}=\max\bigl\{
 \|B\|_{L^\infty_tL^2_z},\ \|B\|_{L^4_{t,z}}\bigr\}.
\]
Now, fix $A\in Y$ and $B\in X'$.  For every decomposition
$A=A_\infty+A_2$, H\"older's inequality gives
\begin{equation*}
\begin{split}
 \|A_\infty B\|_{L^1_tL^2_z}&=\int_J\|A_\infty(t)B(t)\|_2\,dt\leq\int_J\|A_\infty(t)\|_\infty\|B(t)\|_2\,dt\leq\|A_\infty\|_{L^1_tL^\infty_z}
       \|B\|_{L^\infty_tL^2_z}.
\end{split}
\end{equation*}
We also have
\[
 \|A_2B\|_{L^{4/3}_{t,z}}
 \le\|A_2\|_{L^2_{t,z}}\|B\|_{L^4_{t,z}}.
\]
These
bounds imply
\[
 \|AB\|_X\le
 \bigl(\|A_\infty\|_{L^1_tL^\infty_z}
       +\|A_2\|_{L^2_{t,z}}\bigr)\|B\|_{X'}.
\]
Taking the infimum over decompositions of $A$ proves
\[
 \|AB\|_X\le\|A\|_Y\|B\|_{X'}.
 \tag{1.4}\label{eq:XY-multiplier}
\]
\noindent\emph{Step 2: normalization of the interval and the strip width.} By a translation and rescaling, we can assume without loss of generality that the time interval $J$ is equal to $[0,1]$. Indeed, suppose we originally have $J=[t_-,t_+]$ where $L\coloneq t_+-t_->0$. Set
\[
 \widetilde Z(\tau,\xi)=L^{1/2}Z(t_-+L\tau,L^{1/2}\xi),
 \qquad
 \widetilde W(\tau,\xi)=L W(t_-+L\tau,L^{1/2}\xi).
\]
Then
\[
 (i\partial_\tau+\Delta_\xi)\widetilde Z
 =\widetilde W\widetilde Z,
 \qquad 0<\tau<1.
\]
For any decomposition $A=A_\infty+A_2$ of a potential or restricted
potential, write
$\widetilde A_k(\tau,\xi)=L A_k(t_-+L\tau,L^{1/2}\xi)$ for
$k\in\{\infty,2\}$.  The change of variables gives
\begin{align*}
 \|\widetilde Z\|_{L^\infty_\tau L^2_\xi}
 &=\|Z\|_{L^\infty_tL^2_z},
 &\|\widetilde Z\|_{L^4_{\tau,\xi}}
 &=\|Z\|_{L^4_{t,z}},\\
 \|\widetilde A_\infty\|_{L^1_\tau L^\infty_\xi}
 &=\|A_\infty\|_{L^1_tL^\infty_z},
 &\|\widetilde A_2\|_{L^2_{\tau,\xi}}
 &=\|A_2\|_{L^2_{t,z}}.
\end{align*}
Therefore, all of the hypotheses are preserved, with $b$ and $h$ replaced by $b/L^{1/2}$ and $h/L^{1/2}$. It thus suffices to work on $J=[0,1]$. We also suppose without loss of generality that $h \in (0,1].$
\\
\noindent\emph{Step 3: the Carleman inequality.}
Put $\HH_y=i\partial_t+\Delta_y$.  For $\beta\ge2$, define
\[
 \phi_\beta(r)=
 \begin{cases}
  \beta r,&r\ge-1,\\
  r-\beta+1,&r<-1.
 \end{cases}
\]
In particular, $\phi_\beta(r)=\beta r$ on $[0,1]$.  On the other side of
zero we have the bounds
\[
 e^{\phi_\beta(r)}\le1,
 \qquad
 (1+|r|)e^{\phi_\beta(r)}\le1
 \quad(r\le0,\ \beta\ge2).
 \label{eq:IK-weight-left}\tag{1.4a}
\]
We recall the following Carleman estimate proved in \cite{IonescuKenig}.
\begin{lemma}[Ionescu--Kenig; see \cite{IonescuKenig}]
There are absolute constants $C_0\ge1$ and $c_0>0$ such that, for every
$\beta\ge2$ and every $U$ satisfying
\[
 U\in C([0,1];L^2),\qquad
 U=0\text{ for }y_1>1,\qquad
 (1+|y_1|)^{-1}\HH_yU\in X,
\]
one has
\begin{align*}
 &\|\1_{\{0<y_1<1\}}e^{\phi_\beta(y_1)}U\|_{X'}
 +\beta^{c_0}
  \|\1_{\{0<y_1<1\}}e^{\phi_\beta(y_1)}U\|_{L^1_tL^2_y}\\
 &\le C_0\Bigl(
 \|e^{\phi_\beta(y_1)}\HH_yU\|_X
 +\|e^{\phi_\beta(y_1)}U(0)\|_2
 +\|e^{\phi_\beta(y_1)}U(1)\|_2\Bigr).
 \tag{1.4b}\label{eq:IK-bounded-weight}
\end{align*}
\end{lemma}
This is the case $N=1$ of \cite[Lemma~9.1]{IonescuKenig}, with its
exponents chosen as $(p_0,q_0)=(4/3,4/3)$ in dimension two.
The first term on the left will absorb the contribution of $W$
on the strip where its norm is small. The second term will control
the additional potential produced by the change of variables below,
restricted to the strip where it is bounded. The factor
$\beta^{c_0}$ allows this contribution to be absorbed by taking
$\beta$ sufficiently large. Fix from now on the smallness constant
\[
 \varepsilon_*=(4C_0)^{-1}.
\]
\\
\noindent\emph{Step 4: a moving spatial coordinate.}
Set $b_j=b-jh$, and suppose inductively that
$Z=0$ on $[0,1]\times\{z_1>b_j\}$. Fix $0<\delta<1/4$.  Choose a smooth function
$\eta_\delta:[0,1]\to[0,1]$ which is zero on
$[0,\delta]\cup[1-\delta,1]$ and one on $[2\delta,1-2\delta]$.
For example, take
\[
 \eta_\delta(t)=\eta(t/\delta)\eta((1-t)/\delta),
\]
where $\eta:\R\to[0,1]$ is smooth, zero on $(-\infty,1]$, and one on
$[2,\infty)$.  Define
\[
 w_\delta(t)=b_j-h\eta_\delta(t).
\]
Note that $b_{j+1}\le w_\delta\le b_j$, the function $w_\delta$ equals $b_j$
near the time endpoints, and it equals $b_{j+1}$ on the central interval.
Differentiating, we get
\[
 \|w_\delta'\|_\infty\le Ch\delta^{-1},
 \qquad
 \|w_\delta''\|_\infty\le Ch\delta^{-2}.
\]
We now translate the solution and multiply it by a phase:
\[
 \widetilde Z_\delta(t,y)
 =Z(t,y_1+w_\delta(t),y_2),
 \qquad
 U_\delta(t,y)
 =e^{-iw_\delta'(t)y_1/2}\widetilde Z_\delta(t,y).
\]
It is clear that
$U_\delta\in X'$ and
$\|U_\delta\|_{L^\infty_tL^2_y}=\|Z\|_{L^\infty_tL^2_z}$.
Moreover, $U_\delta\in C([0,1];L^2)$. 

The equation for $U_\delta$ includes two additional zeroth-order terms. The motivation for introducing the modulation above is to eliminate the first-order term generated from the translation by $w_{\delta}(t)$. To compute the zeroth order terms, put
$W_\delta(t,y)=W(t,y_1+w_\delta(t),y_2)$.  The chain rule gives
\[
 \HH_y\widetilde Z_\delta
 =W_\delta\widetilde Z_\delta
   +iw_\delta'\partial_{y_1}\widetilde Z_\delta.
\]
Differentiating the phase factor gives, more explicitly,
\begin{align*}
 i\partial_tU_\delta
 &=e^{-iw_\delta'y_1/2}
   \left(i\partial_t\widetilde Z_\delta
         +\tfrac12w_\delta''y_1\widetilde Z_\delta\right),\\
 \Delta_yU_\delta
 &=e^{-iw_\delta'y_1/2}
   \left(\Delta_y\widetilde Z_\delta
         -iw_\delta'\partial_{y_1}\widetilde Z_\delta
         -\tfrac14(w_\delta')^2\widetilde Z_\delta\right).
\end{align*}
The first-order terms cancel when these identities are added.  Hence
\[
 \HH_yU_\delta=(W_\delta+M_\delta)U_\delta,
 \qquad
 M_\delta(t,y)=\tfrac12w_\delta''(t)y_1
               -\tfrac14(w_\delta'(t))^2.
 \tag{1.4c}\label{eq:IK-moving-equation}
\]

We next check the support and forcing hypotheses of
\eqref{eq:IK-bounded-weight}.  If $y_1>h$, then
$y_1+w_\delta(t)>b_j$, so $U_\delta(t,y)=0$ by the induction hypothesis.
Since $h\le1$, this gives the required vanishing for $y_1>1$.
For $t\in[0,\delta]\cup[1-\delta,1]$, we have
$w_\delta=b_j$ and $w_\delta'=0$, so in fact $U_\delta=0$ for $y_1>0$
at those times.  In particular, both time-boundary traces are supported
in $\{y_1\le0\}$.

Define the constant
\[
 K_\delta=\tfrac12\|w_\delta''\|_\infty
          +\tfrac14\|w_\delta'\|_\infty^2.
\]
Note that $|M_\delta(t,y)|\le K_\delta(1+|y_1|)$ everywhere and
$|M_\delta(t,y)|\le K_\delta$ when $0<y_1<1$.
Also,
\[
 W_\delta U_\delta
 =e^{-iw_\delta'y_1/2}(WZ)(t,y_1+w_\delta(t),y_2)\in X,
 \qquad
 \|W_\delta U_\delta\|_X=\|WZ\|_X.
\]
Since
\[
 \|(1+|y_1|)^{-1}M_\delta U_\delta\|_{L^1_tL^2_y}
 \le K_\delta\|U_\delta\|_{L^1_tL^2_y}
 \le K_\delta\|Z\|_{L^\infty_tL^2_z},
\]
equation \eqref{eq:IK-moving-equation} gives
$(1+|y_1|)^{-1}\HH_yU_\delta\in X$.
All the hypotheses of the Carleman inequality have now been verified.

\noindent\emph{Step 5: estimating the right-hand side and absorbing two terms.}
For fixed $j$ and $\delta$,  put
\[
 U_{\delta,\beta}^{+}
 =\1_{\{0<y_1<1\}}e^{\phi_\beta(y_1)}U_\delta,
 \qquad
 A_\beta=\|U_{\delta,\beta}^{+}\|_{X'},
 \qquad
 B_\beta=\|U_{\delta,\beta}^{+}\|_{L^1_tL^2_y}.
\]
Both norms are finite for every $\beta$. Indeed, the weight is at most $e^\beta$
on this strip, and $U_\delta\in X'$.  Notice also that $B_\beta\le A_\beta$
because the time interval has length one.

We first consider the potential term $W_\delta U_\delta$ in the region $y_1>0$.
Whenever $U_\delta(t,y)\ne0$ there, the original spatial coordinate lies in
\[
 b_j-h\le w_\delta(t)<y_1+w_\delta(t)\le b_j.
\]
Thus, only the restriction of $W$ to the fixed strip
$S_j=\{b_j-h<z_1<b_j\}$ enters this part of the product.
More precisely, define
\[
 A_{j,\delta}(t,y)
 =\1_{\{y_1>0\}}(\1_{S_j}W)(t,y_1+w_\delta(t),y_2).
\]
Then $A_{j,\delta}U_{\delta,\beta}^{+}
=\1_{\{0<y_1<1\}}e^{\phi_\beta(y_1)}W_\delta U_\delta$
almost everywhere, and thus
\[
 \|A_{j,\delta}\|_Y\le\|\1_{S_j}W\|_Y\le\varepsilon_*.
\]
The last inequality uses \eqref{eq:strip-smallness} with
$a=b_j-h=b-(j+1)h\le b-h$.
Applying \eqref{eq:XY-multiplier} now gives
\[
 \|\1_{\{0<y_1<1\}}e^{\phi_\beta(y_1)}
       W_\delta U_\delta\|_X
 \le\varepsilon_* A_\beta.
\]
On the same strip, the bound for $M_\delta$ gives
\[
 \|\1_{\{0<y_1<1\}}e^{\phi_\beta(y_1)}
       M_\delta U_\delta\|_X
 \le\|M_\delta U_{\delta,\beta}^{+}\|_{L^1_tL^2_y}
 \le K_\delta B_\beta.
\]
Next, we estimate the source term $\HH_yU_\delta$ in the region $y_1\leq 0$. Here, this expression need not be small. However, it is enough that its
weighted norm be bounded independently of $\beta$.
By \eqref{eq:IK-weight-left} and the restriction property of $X$, we have
\begin{align*}
 \|\1_{\{y_1\le0\}}e^{\phi_\beta(y_1)}W_\delta U_\delta\|_X
 &\le\|W_\delta U_\delta\|_X=\|WZ\|_X,\\
 \|\1_{\{y_1\le0\}}e^{\phi_\beta(y_1)}M_\delta U_\delta\|_X
 &\le K_\delta\|U_\delta\|_{L^1_tL^2_y}
 \le K_\delta\|Z\|_{L^\infty_tL^2_z}.
\end{align*}
The source term also vanishes on $y_1>1$, by
\eqref{eq:IK-moving-equation} and the support of $U_\delta$.
Finally, the support properties of $U_\delta(0)$ and $U_\delta(1)$,
together with \eqref{eq:IK-weight-left}, give
\[
 \|e^{\phi_\beta(y_1)}U_\delta(0)\|_2
 +\|e^{\phi_\beta(y_1)}U_\delta(1)\|_2
 \le\|Z(0)\|_2+\|Z(1)\|_2.
\]
While these boundary terms need not vanish on the whole plane, the above shows that we can estimate their corresponding weighted norms independently of $\beta$. Now set
\[
 D_\delta=\|WZ\|_X+K_\delta\|Z\|_{L^\infty_tL^2_z}
          +\|Z(0)\|_2+\|Z(1)\|_2<\infty.
\]
Substituting the preceding estimates into
\eqref{eq:IK-bounded-weight} gives
\[
 A_\beta+\beta^{c_0}B_\beta
 \le C_0\varepsilon_*A_\beta+C_0K_\delta B_\beta+C_0D_\delta.
 \tag{1.4d}\label{eq:IK-two-absorptions}
\]
Our choice of $\varepsilon_*$ gives $C_0\varepsilon_*=1/4$.
For the fixed $\delta$, take $\beta$ large enough that
$C_0K_\delta\le\beta^{c_0}/2$.
Subtracting the corresponding terms from the two sides of
\eqref{eq:IK-two-absorptions} yields
\[
 \tfrac34 A_\beta+\tfrac12\beta^{c_0}B_\beta
 \le C_0D_\delta,
 \qquad
 A_\beta\le2C_0D_\delta.
 \tag{1.4e}\label{eq:IK-uniform-weighted-bound}
\]
\medskip
\noindent\emph{Step 6: the limits and the induction.}
Fix $0<\alpha<1$.  Since $\phi_\beta(y_1)=\beta y_1$ on $(0,1)$, we have
\[
 \|\1_{\{\alpha<y_1<1\}}U_\delta\|_{L^1_tL^2_y}
 \le e^{-\beta\alpha}B_\beta
 \le 2C_0D_\delta e^{-\beta\alpha}.
\]
For this fixed $\delta$, the constant $D_\delta$ is independent of $\beta$.
Letting $\beta\to\infty$ proves that $U_\delta=0$ almost everywhere on
$[0,1]\times\{\alpha<y_1<1\}$.  Taking the countable sequence
$\alpha=1/m$, $m\ge2$, and recalling that $U_\delta=0$ for $y_1>1$, gives
$U_\delta=0$ almost everywhere where $y_1>0$.
Moreover, the map
$t\mapsto\1_{\{y_1>0\}}U_\delta(t)$ is continuous in $L^2$. Hence,
\[
 Z(t)=0\text{ on }\{z_1>w_\delta(t)\}
 \quad\text{for every }t\in[0,1].
\]
In particular, $Z(t)=0$ on $\{z_1>b_{j+1}\}$ whenever
$2\delta\le t\le1-2\delta$.  Every $t\in(0,1)$ belongs to one such central
interval, by choosing $\delta<\frac12\min\{t,1-t\}$.
Hence the new half-space vanishing holds at every interior time.
Continuity of $t\mapsto\1_{\{z_1>b_{j+1}\}}Z(t)$ extends it to $t=0,1$. This proves the induction step.  Starting at $j=0$, we obtain
\[
 Z(t)=0\text{ on }\{z_1>b-jh\}
 \quad\text{for every }t\in[0,1]\text{ and every integer }j\ge0.
\]
These half-spaces increase to all of $\R^2$.  For each fixed $t$,
monotone convergence therefore gives
\[
 \|Z(t)\|_2^2
 =\lim_{j\to\infty}\int_{\{z_1>b-jh\}}|Z(t,z)|^2\,dz=0.
\]
This proves the desired conclusion.
\end{proof}

We next recover the normal Cauchy data needed to create a half-space solution.

\begin{lemma}
\label{lem:endpoint-wronskian}
Under the hypotheses of Theorem~\ref{thm:time-potential}, on every interval
$I$ with compact closure in $I_0$ one has
\[
 u,v\in L^4(I;L^\infty(\R)),
 \qquad
 \chi u,\chi v\in L^2(I;H^{1/2}(\R))
 \quad(\chi\in C_c^\infty(\R)).
 \tag{1.5}\label{eq:linear-endpoint-regularity}
\]
Moreover, for almost every $t\in I_0$,
\[
 \overline u\,\partial_xu=\overline v\,\partial_xv,
 \qquad
 v\,\partial_xu-u\,\partial_xv=0
 \quad\text{in }\D'(\R_x),
 \tag{1.6}\label{eq:current-wronskian}
\]
where the first product is interpreted by
$H^{-1/2}$--$H^{1/2}$ duality.
\end{lemma}

\begin{proof}
Fix $I=[a,b]$. The endpoint homogeneous and inhomogeneous
Strichartz estimates~\cite{KeelTao} give
\begin{align*}
 \|e^{i(t-a)\partial_x^2}w(a)\|_{L^4_tL^\infty_x(I)}
 &\lesssim\|w(a)\|_2,\\
 \left\|\int_a^t e^{i(t-s)\partial_x^2}F(s)\,ds
 \right\|_{L^4_tL^\infty_x(I)}
 &\lesssim\|F\|_{L^1_tL^2_x+L^{4/3}_tL^1_x}.
\end{align*}
Apply these estimates to the Duhamel formula.  For either solution $w \in \{u,v\}$, this
gives
\begin{align*}
 \|w\|_{L^4_tL^\infty_x(I)}
 \lesssim_I{}&\|w\|_{L^\infty_tL^2_x(I)}
 +\|V_\infty\|_{L^1_tL^\infty_x(I)}
       \|w\|_{L^\infty_tL^2_x(I)}\\
 &+|I|^{1/4}\|V_2\|_{L^2_{t,x}(I\times\R)}
       \|w\|_{L^\infty_tL^2_x(I)}.
\end{align*}
Moreover,
\[
 \|V_2w\|_{L^1_tL^2_x(I)}
 \le |I|^{1/4}\|V_2\|_{L^2_{t,x}(I\times\R)}
                  \|w\|_{L^4_tL^\infty_x(I)}.
\]
Together with the bound for $V_\infty w$, this gives $Vw\in L^1(I;L^2)$.
The local smoothing estimate~\cite{Vega,VegaVisciglia} states that, for
$\chi\in C_c^\infty(\R)$,
\[
 \|\chi e^{it\partial_x^2}g\|_{L^2(I;H^{1/2})}
 \le C_{I,\chi}\|g\|_2.
\]
We then apply this estimate to the homogeneous term in the Duhamel formula.  For the
inhomogeneous term, apply it for each fixed forcing time $s$ and then use
Minkowski's integral inequality.  The two preceding $L^1_tL^2_x$ bounds give
\[
 \left\|\chi\int_a^t e^{i(t-s)\partial_x^2}(Vw)(s)\,ds
 \right\|_{L^2_tH^{1/2}_x}
 \lesssim_{I,\chi}\|Vw\|_{L^1_tL^2_x}.
\]
This proves \eqref{eq:linear-endpoint-regularity}.  For the remainder of the
proof, take a bounded interval $I$ with $\overline I\subset I_0$.

For $w\in H^{1/2}_{\loc}\cap L^\infty_{\loc}$ define
\[
 \langle\overline w\,\partial_xw,\phi\rangle
 :=\langle\partial_xw,\phi\overline w\rangle_{H^{-1/2},H^{1/2}},
 \qquad \phi\in C_c^\infty(\R).
\]
This pairing is well-defined, since after inserting a cutoff equal to one on
$\supp\phi$, multiplication by $\phi$ maps
$H^{1/2}\cap L^\infty$ into $H^{1/2}$, while differentiation maps
$H^{1/2}$ into $H^{-1/2}$. 

With that object being well-defined, the continuity equation
\[
 \partial_t|w|^2+2\partial_x\operatorname{Im}(\overline w\,\partial_xw)=0
 \tag{1.7}\label{eq:continuity-general}
\]
holds in distributions. Let now $\phi\in C_c^\infty(\R)$ and
$\eta\in C_c^\infty(I)$ be real-valued, and set
$a(x)=\int_{-\infty}^x\phi(y)\,dy$. It follows that $a$ is bounded and $a'=\phi$. Although $a$ need not have compact support, it is a bounded Lipschitz function and $a'=\phi$ is compactly supported.  One may therefore test the
 Schr\"odinger equation \eqref{eq:schr-pot} against a regularization of
$\eta(t)a(x)\overline{w(t,x)}$.  Multiplication by $a$ is bounded on
$H^{1/2}$, the potential contribution is real and cancels after taking the
imaginary part, and all remaining pairings are integrable in time by
\eqref{eq:linear-endpoint-regularity}.  Approximation by smooth functions
then gives
\[
 -\int_I\eta'(t)\int_\R a(x)|w(t,x)|^2\,dx\,dt
 =2\int_I\eta(t)
 \langle\operatorname{Im}(\overline w\,\partial_xw),\phi\rangle\,dt.
\]
Equivalently,
\[
 \frac d{dt}\int_\R a|w|^2
 =2\langle\operatorname{Im}(\overline w\,\partial_xw),\phi\rangle
 \tag{1.8}\label{eq:localized-continuity-general}
\]
in $\D'(I)$.  The left-hand sides for $u$ and $v$ agree, since their absolute values agree.  Hence
\[
 \int_I\eta(t)
 \langle\operatorname{Im}(\overline u\,\partial_xu-\overline v\,\partial_xv),\phi\rangle\,dt=0
\]
for every $\eta\in C_c^\infty(I)$ and every spatial test function $\phi$. Now, choose a countable dense family of spatial test functions.  After removing
one null set of times, the current identity holds for every member of this
family and therefore, by continuity of the distributional pairing, for every
$\phi$.  Their real parts agree because the distributional product rule gives
\[
 2\operatorname{Re}(\overline w\,\partial_xw)=\partial_x|w|^2.
\]
This rule follows by approximating $w$ in local $H^{1/2}$ while retaining its
local $L^\infty$ bound, or directly from the Gagliardo product estimate.  We
have proved the first identity in \eqref{eq:current-wronskian} for almost every
time.

We now prove the second identity. First, suppose
$p,q\in H^{1/2}_{\loc}\cap L^\infty_{\loc}$,
$|p|=|q|$, and
$\overline p\,\partial_xp=\overline q\,\partial_xq$ in the sense of distributions. Fix $\phi\in C_c^\infty(\R)$. By multiplying $p$ and $q$ by a
common real-valued smooth cutoff equal to one near $\supp\phi$,
we may assume that $p,q\in H^{1/2}\cap L^\infty$.
This preserves the hypothesis $|p|=|q|$. Define
\[
 \rho=|p|^2=|q|^2,
 \quad s_\varepsilon=\frac{\rho}{\rho+\varepsilon},
 \quad h_\varepsilon=\frac{pq}{\rho+\varepsilon}.
\]
The Gagliardo seminorm~\cite{DiNezzaPalatucciValdinoci} is
\[
 [a]_{H^{1/2}}^2=\iint_{\R^2}
 \frac{|a(x)-a(y)|^2}{|x-y|^2}\,dx\,dy.
\]
By the algebra property of $H^{1/2}\cap L^\infty$ and simple estimates using the seminorm above, we have 
$h_\varepsilon$, $s_\varepsilon p$, and $s_\varepsilon q\in H^{1/2}\cap L^\infty$.  We next observe that we can use $h_\varepsilon\phi$ as a test function. Indeed, by approximating it by smooth, compactly supported functions converging in
$H^{1/2}$ and almost everywhere, the Gagliardo formula and dominated
convergence show that their products with $\overline p$ and $\overline q$
converge in $H^{1/2}$ as well, so the identity extends to this test function.
Since
\[
 \overline p h_\varepsilon=s_\varepsilon q,
 \qquad
 \overline q h_\varepsilon=s_\varepsilon p,
\]
testing the identity $\overline p\,\partial_xp=\overline q\,\partial_xq$
against $h_\varepsilon\phi$ gives
\[
 \langle\partial_xp,s_\varepsilon q\phi\rangle
 -\langle\partial_xq,s_\varepsilon p\phi\rangle=0.
\]
Thus, letting $T_\varepsilon(z)=\varepsilon z/(|z|^2+\varepsilon)$,
\[
 \langle q\,\partial_xp-p\,\partial_xq,\phi\rangle
 =\langle\partial_xp,T_\varepsilon(q)\phi\rangle
  -\langle\partial_xq,T_\varepsilon(p)\phi\rangle.
\]
The maps $T_\varepsilon$ are uniformly Lipschitz, vanish at zero, and satisfy
$|T_\varepsilon(z)|\le\sqrt\varepsilon/2$.  Thus they converge pointwise to
zero.  Moreover,
\[
 \frac{|T_\varepsilon(p(x))-T_\varepsilon(p(y))|^2}{|x-y|^2}
 \le C\frac{|p(x)-p(y)|^2}{|x-y|^2},
\]
with $C$ independent of $\varepsilon$.  The Gagliardo formula and dominated
convergence give
$T_\varepsilon(p),T_\varepsilon(q)\to0$ in $H^{1/2}_{\loc}$.  After
multiplication by the fixed test function $\phi$, duality gives
\begin{align*}
 |\langle\partial_xp,T_\varepsilon(q)\phi\rangle|
 &\le \|\partial_xp\|_{H^{-1/2}(K)}
       \|T_\varepsilon(q)\phi\|_{H^{1/2}},\\
 |\langle\partial_xq,T_\varepsilon(p)\phi\rangle|
 &\le \|\partial_xq\|_{H^{-1/2}(K)}
       \|T_\varepsilon(p)\phi\|_{H^{1/2}},
\end{align*}
where $K$ contains $\supp\phi$.  Both right-hand sides above tend to zero.  It follows thus by passing to
the limit that $q\,\partial_xp-p\,\partial_xq=0$. The second identity in
\eqref{eq:current-wronskian} then follows at once. 
\end{proof}
We are now ready to prove our main result. 
\begin{proof}[Proof of Theorem~\ref{thm:time-potential}]
Fix a bounded interval $I$ with $\overline I\subset I_0$.  On this interval
set $\rho=|u|^2=|v|^2$ and define the exterior product
\[
 F(t,x,y)=u(t,x)v(t,y)-v(t,x)u(t,y).
 \tag{1.9}\label{eq:exterior-product}
\]
We show that $F$ solves a two-dimensional Schr\"odinger equation with potential given by $V(t,x)+V(t,y)$. This can be easily seen by applying
$i\partial_t+\partial_x^2+\partial_y^2$ to each term in the definition of $F$.  For instance, applying it to the first one gives
\begin{align*}
 (i\partial_t+\partial_x^2+\partial_y^2)
   (u(t,x)v(t,y))=V(t,x)u(t,x)v(t,y)+V(t,y)u(t,x)v(t,y).
\end{align*}
The identical calculation for $v(t,x)u(t,y)$ gives
\[
 i\partial_tF+\partial_x^2F+\partial_y^2F=(V(t,x)+V(t,y))F
 \tag{1.10}\label{eq:exterior-equation}
\]
in distributions. On the interval $I$, we have
\[
 F\in C(I;L^2(\R^2))\cap L^4(I\times\R^2).
 \tag{1.11}\label{eq:exterior-class}
\]
Indeed, interpolation between $L^\infty_tL^2_x$ and
$L^4_tL^\infty_x$ gives
\[
 \|w\|_{L^8_tL^4_x}
 \le\|w\|_{L^\infty_tL^2_x}^{1/2}
     \|w\|_{L^4_tL^\infty_x}^{1/2}.
\]
Consequently
\begin{equation*}
\begin{split}
 \|u(t,x)v(t,y)\|_{L^4_{t,x,y}}^4
 &=\int_I\int_\R\int_\R|u(t,x)|^4|v(t,y)|^4\,dx\,dy\,dt\\
 &=\int_I\|u(t)\|_4^4\|v(t)\|_4^4\,dt\\
 &\le\|u\|_{L^8_tL^4_x}^4\|v\|_{L^8_tL^4_x}^4<\infty.
\end{split}
\end{equation*}
The same estimate holds for $v(t,x)u(t,y)$.
Moreover, the right side of \eqref{eq:exterior-equation} belongs to
$L^1(I;L^2(\R^2))$.  For example, Fubini's theorem gives
\begin{align*}
 \|V_2(t,x)u(t,x)v(t,y)\|_{L^2_{x,y}}=\|V_2(t)u(t)\|_{L^2_x}\|v(t)\|_{L^2_y}
 \le\|V_2(t)\|_2\|u(t)\|_\infty\|v(t)\|_2.
\end{align*}
The last expression is integrable in $t$ by Cauchy--Schwarz, because
$V_2\in L^2_tL^2_x$, $u\in L^2_tL^\infty_x$, and
$v\in L^\infty_tL^2_x$.  For $V_\infty$, we use
\[
 \|V_\infty(t,x)u(t,x)v(t,y)\|_2
 \le\|V_\infty(t)\|_\infty\|u(t)\|_2\|v(t)\|_2
\]
and $V_\infty\in L^1_tL^\infty_x$.  The terms with $x$ and $y$ interchanged follow by identical means. Let us now introduce orthogonal coordinates:
\[
 r=\frac{x+y}{\sqrt2},
 \qquad s=\frac{x-y}{\sqrt2}.
\]
Then
\[
 i\partial_tF+\partial_r^2F+\partial_s^2F=QF,
 \qquad
 Q(t,r,s)=V\!\left(t,\frac{r+s}{\sqrt2}\right)
          +V\!\left(t,\frac{r-s}{\sqrt2}\right).
 \tag{1.12}\label{eq:rotated-exterior}
\]
Interchanging $x$ and $y$ changes the sign of $F$.  In the new coordinates
this interchange fixes $r$ and sends $s$ to $-s$.  Therefore
\[
 F(t,r,-s)=-F(t,r,s).
\]
We claim that both of its Cauchy traces at $s=0$ vanish.  For $\psi\in C_c^\infty(I\times\R_r)$ set
\[
 H_\psi(s)=\langle F(\cdot,\cdot,s),\psi\rangle_{t,r}.
\]
Cauchy--Schwarz in $(t,r)$ and Fubini's theorem give
$H_\psi\in L^2_{\loc}(\R_s)$.  Pairing
\eqref{eq:rotated-exterior} with $\psi$ and integrating by parts in $t$ and
$r$ gives, in distributions of $s$,
\[
 H_\psi''(s)
 =\langle QF,\psi\rangle_{t,r}
 +\langle F,i\partial_t\psi-\partial_r^2\psi\rangle_{t,r}.
 \tag{1.12a}\label{eq:scalar-trace-equation}
\]
For a compact $s$-interval, Minkowski's inequality bounds the first term in
$L^2_s$ by a constant times $\|QF\|_{L^1_tL^2_{r,s}}$; the second is bounded
by $\|F\|_{L^1_tL^2_{r,s}}$.  Thus $H_\psi''\in L^2_{\loc}(\R_s)$.  Therefore, we have
$H_\psi\in H^2_{\loc}(\R_s)$. By Sobolev embeddings, we therefore have $H_{\psi}\in C^1$, so both traces at $s=0$ exist. We now compute these two traces. First, the oddness of $F$ directly gives
$H_\psi(0)=0$. To compute $H'_{\psi}(0)$, let $z=(r-s)/\sqrt2$. We have
\begin{align*}
 H_\psi(s)=\sqrt2\iint
 &[u(t,z+\sqrt2s)v(t,z)-v(t,z+\sqrt2s)u(t,z)]\psi(t,\sqrt2z+s)\,dz\,dt.
\end{align*}
For $p\in H^{1/2}$, we also recall the simple property
\[
 \frac{p(\cdot+h)-p}{h}\longrightarrow \partial_xp
 \quad\text{in }H^{-1/2}.
\]
Indeed, on the Fourier-transform side the multiplier is
$(e^{2\pi ih\xi}-1)/h$, which converges pointwise to $2\pi i\xi$ and is bounded by
$2\pi|\xi|$.  Dominated convergence with the $H^{1/2}$ weight proves the claim.
After inserting cutoffs supported near the projection of $\supp\psi$, pair
this convergence with the other local $H^{1/2}$ factor.  The difference
quotients are uniformly bounded from $H^{1/2}$ to $H^{-1/2}$, while the two
local $H^{1/2}$ norms belong to $L^2_t$.  Cauchy--Schwarz in time therefore
justifies passage to the limit under the temporal integral.  The term obtained by
differentiating $\psi(t,\sqrt2z+s)$ vanishes at $s=0$ because the bracket in
the preceding equation vanishes there.  Lemma~\ref{lem:endpoint-wronskian}
then gives
\[
 H_\psi'(0)=2\int_I
 \bigl\langle v\,\partial_xu-u\,\partial_xv,
 \psi(t,\sqrt2\,\cdot)\bigr\rangle\,dt=0.
\]
Combining this with $H_{\psi}(0)=0$, we find that the zero extension
\[
 G(t,r,s)=\1_{\{s>0\}}F(t,r,s)
\]
solves the same equation as $F$ on all of $I\times\R^2$ with the same potential.  Indeed, one can compute the distributional jump formula,
\[
 \partial_s^2(\1_{\{s>0\}}F)
 =\1_{\{s>0\}}\partial_s^2F
  +\delta_{s=0}(\partial_sF)|_{s=0}+\delta'_{s=0}F|_{s=0}.
\]
Both boundary distributions vanish by the two trace identities proved above.  Therefore, it follows that $G$ satisfies
\[
 i\partial_tG+\partial_r^2G+\partial_s^2G=QG.
\]
Clearly $G$ inherits the bound \eqref{eq:exterior-class} from $F$. Moreover,
$QG\in L^1_tL_x^2$ by the estimate already proved for $QF$. It remains to apply Proposition~\ref{prop:strip-propagation}.  Choose a compact
subinterval $J\subset I$ with sufficiently small length, so that
\[
 2\|V_\infty\|_{L^1(J;L^\infty)}<\varepsilon_*/2.
\]
For $S_{a,h}=\{(r,s):a<s<a+h\}$, Fubini's theorem and the substitution
$x=(r+s)/\sqrt2$ give
\begin{equation}\label{eq:ridge-strip}
\begin{split}
 \left\|\1_{S_{a,h}}
 V_2\!\left(t,\frac{r+s}{\sqrt2}\right)\right\|_{L^2_{t,r,s}}^2&=\int_J\int_a^{a+h}\int_\R
 \left|V_2\!\left(t,\frac{r+s}{\sqrt2}\right)\right|^2dr\,ds\,dt\\
 &=\sqrt2h\int_J\int_\R|V_2(t,x)|^2\,dx\,dt
\\
&=\sqrt2h\|V_2\|_{L^2(J\times\R)}^2
\end{split}
\end{equation}
and the same identity holds for the second ridge.  Choose $h>0$ so that their
combined $L^2$ norm is below $\varepsilon_*/2$.  After reflecting $s$ to make
the zero half-space $z_1>0$, \eqref{eq:ridge-strip} and
the hypotheses on the potential show that \eqref{eq:strip-smallness} holds.
Proposition~\ref{prop:strip-propagation} then gives at once that $G=0$.  By oddness, it follows then that $F=0$ almost everywhere on $J\times\R^2$.  Finally, since 
$F\in C(J;L^2(\R^2))$, it follows that $F(t)=0$ in $L^2(\R^2)$ for every
$t\in J$.

Choose now $t_0\in J$. By continuity into $L^2$, the equality
$|u(t)|=|v(t)|$ holds for every $t\in I_0$. If $u(t_0)=0$, then
$v(t_0)=0$, and we set $\zeta=1$.

Suppose instead that $u(t_0)\ne0$. The set on which $u(t_0,x)\ne0$
has positive measure. By Fubini's theorem, we may therefore choose $x_0$
in this set such that $|u(t_0,x_0)|=|v(t_0,x_0)|$ and
$F(t_0,x_0,y)=0$ for almost every $y$. Then
\[
 v(t_0,y)=\frac{v(t_0,x_0)}{u(t_0,x_0)}u(t_0,y)
 \quad\text{for almost every }y.
\]
Thus $v(t_0)=\zeta u(t_0)$ for a scalar $\zeta$, and equality of the
moduli at $x_0$ gives $|\zeta|=1$.

We now recall that if two mild solutions have the same value at
one endpoint of a sufficiently short interval $K$, their difference $w$
satisfies the homogeneous Duhamel equation.  The homogeneous and inhomogeneous
Strichartz estimates, together with the estimates used in
Lemma~\ref{lem:endpoint-wronskian}, give
\[
 \|w\|_{L^\infty_tL^2_x(K)\cap L^4_tL^\infty_x(K)}
 \le C\bigl(
 \|V_\infty\|_{L^1_tL^\infty_x(K)}
 +|K|^{1/4}\|V_2\|_{L^2_{t,x}(K)}\bigr)
 \|w\|_{L^\infty_tL^2_x(K)\cap L^4_tL^\infty_x(K)}.
\]
Hence, for $|K|$ small enough, $w=0$ on $K$.
Partitioning each compact subinterval of $I_0$ into finitely many such pieces
extends uniqueness forward and backward from $t_0$.  Since the Schr\"odinger equation with a real-valued potential is invariant under multiplication by a constant phase, $v$ and $\zeta u$ have the same
initial value at $t_0$ and therefore agree throughout $I_0$.
\end{proof}

\section{Cubic NLS: Proof of Corollary \ref{cor:cubic-NLS}}

Next, we prove Corollary~\ref{cor:cubic-NLS}, which is to establish phase retrieval in $L^2(\mathbb{R})$ for the nonlinear cubic Schr\"odinger equation 
\[
 i\partial_t w+\partial_x^2w=\sigma|w|^2w.
 \tag{2.1}\label{eq:cubic-NLS}
\]
\begin{proof}[Proof of Corollary~\ref{cor:cubic-NLS}]
Our main goal will be to verify that the equation, as well as the `potential' $V = \sigma|u|^2 = \sigma|v|^2$, satisfy the hypotheses of Theorem \ref{thm:time-potential}. This is entirely standard in the theory of nonlinear dispersive PDE (see \cite{Tao2006}), but we provide the short argument for convenience to the reader, whose main research focus may not be in PDE. Fix then a bounded interval $I$. On a time interval $K$, the homogeneous and
inhomogeneous Strichartz estimates~\cite{KeelTao} give
\[
 \left\|e^{i(t-t_0)\partial_x^2}w(t_0)
 -i\sigma\int_{t_0}^te^{i(t-s)\partial_x^2}|w(s)|^2w(s)\,ds
 \right\|_{L^\infty_tL^2_x\cap L^8_tL^4_x}
 \lesssim \|w(t_0)\|_2
 +|\sigma|\,\||w|^2w\|_{L^{8/7}_tL^{4/3}_x}.
\]
For the nonlinear term, H\"older's inequality in space and the finite-measure
embedding in time give
\begin{align*}
 \||w|^2w\|_{L^{8/7}_tL^{4/3}_x(K)}\le |K|^{1/2}
       \||w|^2w\|_{L^{8/3}_tL^{4/3}_x(K)}=|K|^{1/2}\|w\|_{L^8_tL^4_x(K)}^3.
 \tag{2.2}\label{eq:NLS-fixed-point-estimate}
\end{align*}
The pointwise inequality
\[
 \bigl||z_1|^2z_1-|z_2|^2z_2\bigr|
 \le C(|z_1|^2+|z_2|^2)|z_1-z_2|
\]
and the same H\"older estimate now yield
\begin{align*}
 &\||w_1|^2w_1-|w_2|^2w_2\|_{L^{8/7}_tL^{4/3}_x(K)}\le C|K|^{1/2}
 \bigl(\|w_1\|_{L^8_tL^4_x(K)}^2+\|w_2\|_{L^8_tL^4_x(K)}^2\bigr)
 \|w_1-w_2\|_{L^8_tL^4_x(K)}.
\end{align*}
Thus the Duhamel map is a contraction on a ball in
$C(K;L^2)\cap L^8(K;L^4)$ when $K$ is sufficiently short.  This is the standard local
$L^2$ construction for the solutions considered in the corollary.  For smooth solutions,
multiplication by $\overline w$, integration in $x$, and taking imaginary
parts give
\[
 \frac d{dt}\|w(t)\|_2^2=0.
\]
Approximation of the initial data and the $C_tL^2_x$ convergence in the fixed
point construction extend mass conservation to $L^2$ solutions.  Since the
length of the contraction interval depends only on $|\sigma|\|w(t_0)\|_2^2$,
the local solution can be iterated for all time.  In particular, covering the
bounded interval $I$ by finitely many local intervals gives
\[
 u,v\in L^8(I;L^4(\R)).
 \tag{2.3}\label{eq:NLS-L8L4}
\]
The inhomogeneous term belongs to $L^{8/7}_tL^{4/3}_x$ by
\eqref{eq:NLS-fixed-point-estimate}.  The pair $(8,4)$ is admissible, so its dual
pair is $(8/7,4/3)$.  The inhomogeneous Strichartz estimate from this dual
pair to the admissible endpoint $(4,\infty)$ gives
\[
 u,v\in L^4(I;L^\infty(\R)).
 \tag{2.4}\label{eq:NLS-endpoint}
\]
Finally, \eqref{eq:NLS-L8L4} and the finite length of $I$ imply
\[
 \|u\|_{L^4(I\times\R)}
 \le |I|^{1/8}\|u\|_{L^8_tL^4_x},
\]
and the same estimate holds for $v$. On $I_0\times\R$, set $\rho=|u|^2=|v|^2$ and
$V=\sigma\rho$.  For every bounded $J$ with $\overline{J}\subset I_0$, the preceding estimate
gives
\[
 \|V\|_{L^2(J\times\R)}
 =|\sigma|\|u\|_{L^4(J\times\R)}^2<\infty.
\]
Because $\sigma$ and $\rho$ are real, $V$ is real.  On $I_0$ the two
nonlinear equations are therefore the same linear equation:
\[
 i\partial_tu+\partial_x^2u=Vu,
 \qquad
 i\partial_tv+\partial_x^2v=Vv.
\]
Theorem~\ref{thm:time-potential} gives $v=\zeta u$ on $I_0$ for some
$\zeta\in\T$.  Choose $t_0\in I_0$.  The fixed-point difference estimate
above gives uniqueness for the cubic equation in
$C_tL^2_x\cap L^8_tL^4_x$.  Since the two nonlinear solutions agree up to
$\zeta$ at $t_0$, phase-rotation invariance of NLS and uniqueness give
$v=\zeta u$ for every time.
\end{proof}

\section{Failure in dimensions at least two: Proof of Theorem \ref{thm:higher-dimensional-failure}}

As a final contribution, we build a simple family of counterexamples showing Theorem~\ref{thm:higher-dimensional-failure}. This provides a \emph{negative} answer to the higher-dimensional assertion in \cite{Jaming2025}. 

\begin{proof}[Proof of Theorem~\ref{thm:higher-dimensional-failure}]
We first prove the assertion for $d=2$.  Let
\[
 \phi(x)=e^{-x^2/2},
 \qquad \psi(x)=-\phi'(x)=xe^{-x^2/2},
\]
 define
\begin{align*}
 a(x_1,x_2)=\phi(x_1)\psi(x_2),\hspace{5mm}b(x_1,x_2)=\psi(x_1)\phi(x_2),
\end{align*}
and also define $f=a+ib$ and $g=a-ib = \overline{f}$. Since
\[
 f(x_1,x_2)=e^{-(x_1^2+x_2^2)/2}(x_2+ix_1),
 \qquad
 g(x_1,x_2)=e^{-(x_1^2+x_2^2)/2}(x_2-ix_1),
\]
both functions are Schwartz.  On the other hand, they are clearly linearly independent. The free propagator preserves tensor products.  Indeed, the Fourier transform
of $p(x_1)q(x_2)$ is $\widehat p(\xi_1)\widehat q(\xi_2)$, while
\[
 e^{-4\pi^2it(\xi_1^2+\xi_2^2)}=e^{-4\pi^2it\xi_1^2}e^{-4\pi^2it\xi_2^2}.
\]
Therefore
\[
 e^{it\Delta_2}(p\otimes q)
 =(e^{it\partial_x^2}p)\otimes(e^{it\partial_x^2}q).
\]
Define
\[
 A_t(x)=(1+2it)^{-1/2}
 \exp\!\left(-\frac{x^2}{2(1+2it)}\right),
\]
where the square root is chosen continuously from $t=0$.  Direct
differentiation gives
\[
 i\partial_tA_t+\partial_x^2A_t=0,
 \qquad A_0=\phi.
\]
Uniqueness of the free evolution yields $e^{it\partial_x^2}\phi=A_t$.
Since $\psi=-\phi'$ and spatial differentiation commutes with the free
propagator, we have
\[
 e^{it\partial_x^2}\psi
 =-\partial_xA_t=\frac{x}{1+2it}A_t.
\]
Consequently
\begin{align*}
 e^{it\Delta}f(x_1,x_2)
 &=A_t(x_1)A_t(x_2)\frac{x_2+ix_1}{1+2it},\\
 e^{it\Delta}g(x_1,x_2)
 &=A_t(x_1)A_t(x_2)\frac{x_2-ix_1}{1+2it}.
\end{align*}
Since $x_1,x_2$ are real,
\[
 |x_2+ix_1|^2=x_1^2+x_2^2=|x_2-ix_1|^2.
\]
All other factors in the two formulas are identical.  Their moduli therefore
coincide at every point and every time.

For $d>2$, we can tensor both functions with
$\Phi(x_3,\ldots,x_d)=\exp(-\frac12\sum_{j=3}^dx_j^2)$.  The free propagator
again factors, so both the modulus identity and linear independence are preserved. This finishes the proof.
\end{proof}

\section{Comments on alternative proof methods}
\label{sec:stationary-comments}

We also provide a companion note that contains three alternative proofs for time-independent potentials. They are recorded there because their hypotheses and methods differ from the time-dependent argument of Theorem~\ref{thm:time-potential}. However, we give the interested reader a quick panorama of the techniques used in those cases below. 

\subsection{Masuda's unique-continuation theorem}

Assume that $V$ is time independent and that the two solutions belong to
$C(I;H^1(\R))$.  For
$\rho_w=|w|^2$ and
$j_w=\operatorname{Im}(\overline w\,\partial_xw)$, the equation with real
potential gives
\[
 \partial_t\rho_w+2\partial_x j_w=0.
\]
Equality of the densities therefore implies
$\partial_x(j_u-j_v)=0$.  The $H^1$ assumption gives
$\|j_w(t)\|_1\le \|w(t)\|_2\|\partial_xw(t)\|_2$, so the constant in $x$
must be zero.  Since
$2\operatorname{Re}(\overline w\,\partial_xw)=\partial_x|w|^2$, the real
parts agree as well, and hence
\[
 \overline u\,\partial_xu=\overline v\,\partial_xv
 \quad\text{almost everywhere on }I\times\R.
\]
The $H^1$ regularity also provides jointly continuous representatives.  On
the open set $\Omega=\{(t,x):u(t,x)\ne0\}$, define $q=v/u$.  Then $|q|=1$
and
\[
 0=\overline v\,\partial_xv-\overline u\,\partial_xu
   =|u|^2\overline q\,\partial_xq.
\]
Thus $q$ is constant in $x$ on each sufficiently small rectangle contained
in $\Omega$.  No division at a zero of $u$ is used.

The remaining point is to remove a possible dependence on time.  On such a
rectangle choose a nonzero, nonnegative $\chi\in C_c^\infty(\R)$ supported
in its spatial side, and write
\[
 A(t)=\int_\R\chi\overline u v\,dx,
 \qquad B(t)=\int_\R\chi|u|^2\,dx.
\]
If $v=c(t)u$ on the rectangle, then $A=cB$ and $B>0$.  Testing the two
equations and integrating once in $x$ gives $A'=cB'$ almost everywhere.
Consequently $c'B=0$, so $c$ is constant.  The quotient is therefore
constant on each connected component of $\Omega$.  For one such component,
$w=v-\zeta u$ is a solution which vanishes on a nonempty open subset of
spacetime.  Masuda's theorem~\cite{Masuda} yields $w\equiv0$.

This is the shortest of the three stationary proofs.  It was not used for the
main theorem because Masuda's result requires a time-independent semibounded
Hamiltonian, while the common potential arising in a nonlinear equation is
generally time dependent.  The phase reduction also uses $H^1$ regularity on
the observation interval.  Theorem~\ref{thm:time-potential}, by contrast,
allows the stated time-dependent potential class and arbitrary $L^2$ data.

\subsection{Faddeev scattering}

We write
\[
 L^1_1(\R)=\left\{V\in L^1(\R):\int_\R(1+|x|)|V(x)|\,dx<\infty\right\}.
\]
For a real potential $V\in L^1_1(\R)$, one-dimensional scattering
theory~\cite{DeiftTrubowitz,EgorovaKopylovaMarchenkoTeschl} shows that
$H=-\partial_x^2+V$ has finitely many simple negative eigenvalues
$-\kappa_j^2$, while its absolutely continuous spectrum on $(0,\infty)$
has multiplicity two.  The corresponding spectral expansion has the
form
\[
 e^{-itH}f
 =\sum_j c_j e^{it\kappa_j^2}\phi_j
  +\sum_{\alpha=1}^2\int_0^\infty
    e^{-itk^2}g_\alpha(k)e_\alpha(\,\cdot\,,k)\,dk.
\]
Here, $\phi_j$ is an $L^2$ eigenfunction, and, for each $k>0$, the functions
$e_1(\,\cdot\,,k)$ and $e_2(\,\cdot\,,k)$ form a basis of the
two-dimensional spectral fiber at energy $k^2$.  Thus $\alpha\in\{1,2\}$
indexes the two components of the absolutely continuous spectral
representation.  The expansion remains valid if zero energy is resonant,
meaning that $H\phi=0$ has a nonzero bounded solution which is not in $L^2$.
The coefficients $c_j$ and the two functions $g_\alpha$ give a unitary
spectral representation of arbitrary $L^2$ initial data.  We denote this
unitary spectral transform by $\mathcal F$.

The temporal Fourier transform of the density identity separates products by
their energy difference $\delta=E-F$.  Along a fixed nonzero gap, we use the
energy sum $s=E+F$.  If $y$ and $z$ solve the stationary equations at energies
$E$ and $F$, their products satisfy
\[
 A'=B,\qquad B'=(2V-s)A+2D,\qquad
 C'=-\delta A,\qquad
 D'=\left(V-\frac{s}{2}\right)B+\frac{\delta}{2}C,
\]
where $A=yz$, $B=y'z+yz'$, $C=y'z-yz'$, and $D=y'z'$.  The system contains
no derivative of $V$.  Moreover, when $E\ne F$, the four products obtained
from bases of the two two-dimensional solution spaces are linearly
independent: if their scalar component $A$ vanishes, the system successively
forces $B$, $D$, and $C$ to vanish.

To separate the energy sum, one divides a fixed-gap relation by $s-\lambda$
and integrates in $s$.  The resulting vector solves the same product system
with sum $\lambda$.  For sufficiently negative $\lambda$, both auxiliary
energies $(\lambda\pm\delta)/2$ are negative and the lower one lies below
$\inf\sigma(H)$.  Scattering asymptotics and the Riemann--Lebesgue lemma show
that its scalar component decays at both spatial infinities.  The behavior of
solutions to the negative-energy ODE then forces the whole vector to vanish:
a nonzero solution decaying at both ends at the lower energy would be an
$L^2$ eigenfunction below the spectrum.  Analyticity in $\lambda$ and
Stieltjes inversion show that the original fixed-gap measure is zero.
Applying this for almost every $\delta\ne0$, and treating separately the
finite sum of products of two eigenfunctions, recovers
\[
 (\mathcal F u_0)\otimes\overline{\mathcal F u_0}
 =(\mathcal F v_0)\otimes\overline{\mathcal F v_0}.
\]
Equality of these rank-one operators gives $v_0=\zeta u_0$.

This proof allows arbitrary $L^2$ data, negative eigenvalues, and a
zero-energy resonance.  It was left in the companion note because it requires
a time-independent potential in $L^1_1(\R)$, observations for all
$t\in\R$, and the complete spectral resolution of $H$.  As none of these
spectral reductions is available for the time-dependent potentials arising
in the nonlinear application, the argument is hence more specialized
than the physical-space proof of Theorem~\ref{thm:time-potential}.

\subsection{Observation on one exterior half-line}

Suppose that $V$ is smooth, compactly supported, and vanishes on
$(R,\infty)$.  Each negative-energy eigenfunction is then a nonzero multiple
of $e^{-\kappa_j(x-R)}$ on this half-line.  The two components of the
absolutely continuous spectral representation can also be combined into one
Fourier integral by assigning them to positive and negative $k$, respectively,
so the restriction
of an evolution has the form
\[
 e^{-itH}f(x)
 =\sum_j a_j e^{-\kappa_j(x-R)}e^{it\kappa_j^2}
  +\int_\R F(k)e^{ik(x-R)-ik^2t}\,dk.
\]
Thus the half-line problem first reduces to phase retrieval for a free
Fourier extension coupled to finitely many decaying exponentials.

Test the density difference in time against the inverse Fourier transform of
$\chi\in C_c^\infty(\R)$, and take the Laplace transform in $y=x-R>0$.
If
\[
 q(k,\ell)=F(k)\overline{F(\ell)}
           -G(k)\overline{G(\ell)},
\]
then the free--free term becomes
\[
 \Phi_\chi(z)=\int_\R\frac{M_\chi(\xi)}{z-i\xi}\,d\xi,
 \qquad
 M_\chi(\xi)=\int_\R q(k,k-\xi)
                  \chi(\xi^2-2k\xi)\,dk,
 \qquad \operatorname{Re}z>0.
\]
The compact support of $\chi$ and a Schur estimate make $M_\chi$ integrable,
so this is an ordinary Cauchy transform.

Every term containing a negative-energy eigenfunction is holomorphic in a
strip across the imaginary axis.  The transformed density identity therefore continues
$\Phi_\chi$ through that axis.  The Cauchy jump formula gives
$M_\chi=0$ for every $\chi$.  The change of variables
$\xi=k-\ell$, $\omega=\ell^2-k^2$ then yields $q=0$ away from the diagonal,
which is a null set.  Hence $F\otimes\overline F=G\otimes\overline G$ and
$G=\zeta F$.  The mixed terms recover the negative-eigenvalue coefficients by
successively crossing the lines $\operatorname{Re}z=-\kappa_j$; when the
absolutely continuous part vanishes, two elementary Vandermonde arguments
recover the finite rank-one matrix of negative-eigenvalue coefficients.
Finally, the transmission coefficient has no positive-energy
zeros~\cite{DeiftTrubowitz}, so the relation on the free half-line recovers
both components of the original absolutely continuous spectral data.

This theorem uses less spatial data than Theorem~\ref{thm:time-potential},
but it assumes a time-independent, smooth, compactly supported potential and
observations for every time.  Its proof depends on an exactly free exterior
region, on analytic continuation of the associated Laplace transforms, and
on one-dimensional scattering coefficients.  These ingredients do not
persist for the time-dependent potentials in the nonlinear problem, the reason for which we decided to leave it out.

\end{document}